\documentclass[11pt]{article}

\usepackage[margin=1.3in]{geometry}

\usepackage{amssymb,amsmath}
\usepackage{amsthm}
\usepackage{hyperref}
\usepackage{mathrsfs}
\usepackage{textcomp}
\hypersetup{
    colorlinks,%
    citecolor=black,%
    filecolor=black,%
    linkcolor=black,%
    urlcolor=black
}

\usepackage{verbatim}

\usepackage{sectsty}

\makeatletter

\newdimen\bibspace
\renewenvironment{thebibliography}[1]{%
 \section*{\refname 
       \@mkboth{\MakeUppercase\refname}{\MakeUppercase\refname}}%
     \list{\@biblabel{\@arabic\c@enumiv}}%
          {\settowidth\labelwidth{\@biblabel{#1}}%
           \leftmargin\labelwidth
           \advance\leftmargin\labelsep
           \itemsep\bibspace
           \parsep\z@skip     %
           \@openbib@code
           \usecounter{enumiv}%
           \let\p@enumiv\@empty
           \renewcommand\theenumiv{\@arabic\c@enumiv}}%
     \sloppy\clubpenalty4000\widowpenalty4000%
     \sfcode`\.\@m}
    {\def\@noitemerr
      {\@latex@warning{Empty `thebibliography' environment}}%
     \endlist}

\makeatother

\makeatletter

\newtheorem{thm}{Theorem}

\newtheorem{prop}[thm]{Proposition}

\def\Xint#1{\mathchoice
  {\XXint\displaystyle\textstyle{#1}}%
  {\XXint\textstyle\scriptstyle{#1}}%
  {\XXint\scriptstyle\scriptscriptstyle{#1}}%
  {\XXint\scriptscriptstyle\scriptscriptstyle{#1}}%
  \!\int}
\def\XXint#1#2#3{{\setbox0=\hbox{$#1{#2#3}{\int}$}
  \vcenter{\hbox{$#2#3$}}\kern-.5\wd0}}

\def\dashint{\Xint-}

                \newcommand{\lda}{\lambda}
                
           \newcommand{\ud}{\mathrm{d}}
\newcommand{\be}{\begin{equation}}      \newcommand{\ee}{\end{equation}}

\newcommand{\R}{\mathbb{R}}              \newcommand{\Sn}{\mathbb{S}^n}

\begin{document}

\title{\textbf{A derivation of the sharp Moser-Trudinger-Onofri inequalities from the fractional Sobolev inequalities}
\bigskip}

\author{Jingang Xiong\footnote{Supported in part by NSFC 11501034, NSFC 11571019 and the key project NSFC 11631002.} }


\maketitle

\begin{abstract} We derive the sharp Moser-Trudinger-Onofri inequalities on the standard $n$-sphere and CR $(2n+1)$- sphere as the limit of the  sharp fractional Sobolev inequalities for all $n\ge 1$. On the $2$-sphere and $4$-sphere, this  was established recently by S.-Y. Chang and F. Wang. Our  proof  uses  an alternative and elementary argument.

\end{abstract}

\section{Introduction}

In \cite{O}, E. Onofri proved the sharp Moser-Trudinger inequality on the unit $2$-sphere
\[
\ln \dashint_{\mathbb{S}^2} e^{2w} \,\ud \mu_{g_0} \le \dashint_{\mathbb{S}^2}|\nabla w|^2 \,\ud \mu_{g_0} +2 \dashint_{\mathbb{S}^2}w \,\ud\mu_{g_0}  \quad \mbox{for }w\in W^{1,2}(\mathbb{S}^2),
\]
where $g_0$ is the standard metric and $\dashint_{\mathbb{S}^2} \,\ud \mu_{g_0} = \frac{1}{|\mathbb{S}^2 |} \int_{\mathbb{S}^2}  \,\ud \mu_{g_0}$. Onofri's proof is based on a version of the Moser-Trudinger inequality due to T. Aubin \cite{A}  which holds under the additional constraint $\dashint_{\mathbb{S}^2} e^{2w} x \,\ud \mu_{g_0}=0$, $x\in \R^{3}$; see C. Gui and A. Moradifam \cite{GM} for the proof of sharp form of Aubin's inequality which was conjectured by S.-Y. Chang and P. Yang \cite{CY}. Until now, there have been many different proofs of the Moser-Trudinger-Onofri inequality. A collection of them can be found in the survey J. Dolbeault, M. J. Esteban, and G. Jankowiak \cite{DEJ}. In \cite{R}, Y. Rubinstein gave a K\"ahler geometric proof of the sharp inequality and obtained an optimal extension of it to the higher dimensional K\"ahler-Einstein manifolds.  Rubinstein's proof is based on the earlier results of W. Ding and G. Tian \cite{DT} and G. Tian \cite{T}. On the standard $n$- sphere $\Sn$,  the Moser-Trudinger-Onofri inequality was established  by T. Branson, S.-Y. Chang and P. Yang \cite{BCY} and W. Beckner \cite{Be} for $n=4$, and by \cite{Be} for all $n\ge 1$.

Recently,  S.-Y. Chang and F. Wang \cite{CW} derived the sharp Moser-Trudinger-Onofri inequality on the $2$- and $4$- spheres as the limit case of the fractional power Sobolev inequalities, which was motivated by a dimensional continuation argument of T. Branson. The  proof  of  \cite{CW}
exploits the definition of the fractional order operators as generalized Dirichlet-to-Neumann operators from scattering theory, and uses the extension formula of the fractional order operators, which was first introduced  by  L. Caffarelli and L. Silvestre \cite{CaS} on the Euclidean spaces, and later generalized  to operators defined on the boundaries of asymptotically hyperbolic  manifolds by S.-Y. Chang and M. Gonz\'alez \cite{CG}, and J. Case and S.-Y. Chang \cite{CC}.  In the final remark of \cite{CW}, they commented that it is plausible that their arguments can be applied to other dimensions, but the arguments would become increasingly delicate when $n$ is large.

In this paper,  we derive the sharp Moser-Trudinger-Onofri inequality as the limit case of the fractional power Sobolev inequalities on $\Sn$ for all $n\ge 1$. Instead of using Chang-Wang's argument from scattering theory, our proof uses the explicit formulas of the fractional order operators on the spheres. Chang-Wang's method should have broader applications in related problems on manifolds.  On the dual side, E. Carlen and M. Loss \cite{CL} derived the sharp logarithmic Hardy-Littlewood-Sobolev inequality on $\Sn$ from the sharp HLS inequalities via endpoint differentiation, which in turn implies the sharp Moser-Trudinger-Onofri inequality.

Our argument works in the CR setting, too. In this situation, a sharp Moser-Trudinger-Onofri inequality on CR sphere $\mathbb{S}^{2n+1}$ was discovered by T. Branson, L. Fontana and C. Morpurgo \cite{BFM} after introducing the $\mathcal{A}_Q'$ operator of order $Q=2n+2$. On the other hand, R. Frank and E. Lieb \cite{FL} proved the sharp fractional Sobolev inequalities as a corollary of their sharp HLS inequalities. \cite{FL} also proved the limiting cases of HLS by differentiating  HLS at the endpoints; see Corollary 2.4 and Corollary 2.5.   We derive   the sharp Moser-Trudinger-Onofri inequality of \cite{BFM} as the limit of the sharp fractional Sobolev inequalities of \cite{FL} in a similar way.

In the next section, we extend \cite{CW} to all dimensions $n\ge 1$ by a different approach. In section 3, we prove the analogue in the CR spheres setting.

\bigskip

\noindent\textbf{Acknowledgments:}
The author is  grateful to Professor G. Tian for his kind advice on presentation and for his insightful comments.
He also thanks Professor R. Frank for clarifying the limiting process in the literature.

\medskip

\section{The standard spheres setting}

Let $n\ge 1$, $\Sn\subset \R^{n+1}$ be the unit $n$-dimensional sphere. For $\gamma>0$, let
\[
 P_\gamma=\frac{\Gamma(B+\frac{1}{2}+\gamma)}{\Gamma(B+\frac{1}{2}-\gamma)},\quad B=\sqrt{-\Delta_{g_0}+\left(\frac{n-1}{2}\right)^2},
\]
where $\Delta_{g_0}$ is the Laplace-Beltrami operator on $\Sn$ with respect to the standard induced metric $ g_0$ from $\R^{n+1}$. More precisely, $B$
and $P_\gamma$ are determined
by the formulas
\be\label{eigenvalue}
 B\Big(Y^{(k)}\Big)=\left(k+\frac{n-1}{2}\right)Y^{(k)}\quad\mbox{and}\quad P_\gamma\Big(Y^{(k)}\Big)=\frac{\Gamma(k+\frac{n}{2}+\gamma)}{\Gamma(k+\frac{n}{2}-\gamma)}Y^{(k)}
\ee
for every spherical harmonic $Y^{(k)}$ of degree $k\ge 0$, where $\Gamma(\cdot)$ is the Gamma function.

Let $\gamma\in (0,n/2)$. The sharp Sobolev inequality on $\Sn$ asserts that
\be\label{eq:f-sobolev}
Y(n, \gamma)\left(\dashint_{\mathbb{S}^n}|v|^{\frac{2n}{n-2\gamma}}\,\ud \mu_{g_0} \right)^{\frac{n-2\gamma}{n}}\leq
\dashint_{\mathbb{S}^n}vP_{\gamma}(v)\,\ud \mu_{g_0} \quad \mbox{for }v\in C^\infty(\mathbb{S}^n),
\ee
where $Y(n, \gamma):=\frac{\Gamma(\frac{n}{2}+\gamma)}{\Gamma(\frac{n}{2}-\gamma)}$ and $\dashint_{\mathbb{S}^n} \,\ud \mu_{g_0} =\frac{1}{|\mathbb{S}^n|}\int_{\mathbb{S}^n}\,\ud \mu_{g_0} $. The sharp Moser-Trudinger-Onofri inequality asserts that
\be\label{eq:m-t}
\frac{2(n-1)!}{n}\ln \dashint_{\mathbb{S}^n} e^{nw} \,\ud \mu_{g_0}  \le  \dashint_{\mathbb{S}^n} \Big(w P_{n/2} w + 2(n-1)! w \Big) \,\ud \mu_{g_0}  \quad \mbox{for }w\in C^\infty(\mathbb{S}^n).
\ee
See W. Beckner \cite{Be} for the proofs of the both inequalities. In particular, \eqref{eq:f-sobolev} is a consequence of the sharp HLS inequality due to E. Lieb \cite{Lie83} in the Euclidean spaces.

Recently,  S.-Y. Chang and F. Wang \cite{CW} studied the limit of \eqref{eq:f-sobolev} when $n=2$ and $n=4$. Generalizing the cases  $n=2$ and $n=4$ from \cite{CW}, we have

\begin{prop}\label{prop:main} For $\gamma\in (0,n/2)$ and any $w\in C^{\infty}(\mathbb{S}^n)$, let $v= e^{(\frac{n}{2}-\gamma) w}$. Denote
\[
LHS_\gamma:=  \frac{4}{(n-2\gamma)^2} Y(n, \gamma)  \left[  \left( \dashint_{\mathbb{S}^n}|v|^{\frac{2n}{n-2\gamma}}\,\ud \mu_{g_0} \right)^{\frac{n-2\gamma}{n}}- \dashint_{\mathbb{S}^n}|v|^{2}\,\ud \mu_{g_0}  \right]
\]
and
\[
RHS_\gamma :=  \frac{4}{(n-2\gamma)^2}\left[ \dashint_{\mathbb{S}^n}vP_{\gamma}(v)\,\ud \mu_{g_0} -  Y(n, \gamma)\dashint_{\mathbb{S}^n}|v|^{2}\,\ud \mu_{g_0}   \right].
\]
Then
\be  \label{eq:left}
\lim_{\gamma\to n/2} LHS_\gamma= \frac{2(n-1)!}{n} \ln \dashint_{\mathbb{S}^n} e^{n(w-\bar w)} \,\ud \mu_{g_0}
\ee
and
\be \label{eq:right}
\lim_{\gamma\to n/2} RHS_\gamma= \dashint_{\mathbb{S}^n} w P_{n/2} w  \,\ud \mu_{g_0} ,
\ee
where $\bar w$ is the average of $w$ over $\Sn$.

\end{prop}

Consequently, we immediately  have

\begin{thm} \label{thm} We can derive the sharp Moser-Trudinger-Onofri inequality \eqref{eq:m-t} from the sharp Sobolev inequality  \eqref{eq:f-sobolev} by sending $\gamma\to \frac{n}{2}$.

\end{thm}

\begin{proof}[Proof of Proposition \ref{prop:main}]

The proof of \eqref{eq:left} essentially follows from the proof of Lemma 3.1 of  \cite{CW}.
Note that \begin{align*}
&\left( \dashint_{\mathbb{S}^n}e^{nw}\,\ud \mu_{g_0} \right)^{\frac{n-2\gamma}{n}}- \dashint_{\mathbb{S}^n}e^{(n-2\gamma)w}\,\ud \mu_{g_0} \\&
  = \left( \dashint_{\mathbb{S}^n}e^{nw}\,\ud \mu_{g_0} \right)^{\frac{n-2\gamma}{n}}-1- \dashint_{\mathbb{S}^n}(e^{(n-2\gamma)w}-1)\,\ud \mu_{g_0} .
 \end{align*}
 Then by L'H\^opital's rule
 \begin{align*}
 \lim_{\gamma\to n/2} LHS_\gamma &=2\Gamma(n)  \left(\frac{1}{n} \ln  \dashint_{\mathbb{S}^n}e^{nw}\,\ud \mu_{g_0}  - \dashint_{\mathbb{S}^n} w\,\ud \mu_{g_0}  \right) \\&
 =\frac{2(n-1)!}{n} \ln \dashint_{\mathbb{S}^n} e^{n(w-\bar w)} \,\ud \mu_{g_0} .
 \end{align*}
 Therefore, \eqref{eq:left} is proved.

To prove \eqref{eq:right}, using the Taylor expansion of the exponential function, we write
\[
v= e^{\frac{n-2\gamma}{2} w }= 1+(\frac{n}{2}-\gamma) w+ (n-2\gamma)^2 f,
\] where $f=\frac{1}{8} w^2 \int_0^1(1-s)e^{\frac{n-2\gamma}{2} w s}\,\ud s\in C^\infty(\Sn)$ is uniformly bounded in $C^{2n}$ norm as $\gamma \to n/2$.
Then we see that
\begin{align*}
&\dashint_{\mathbb{S}^n}vP_{\gamma}(v)\,\ud \mu_{g_0}  \\&=\dashint_{\mathbb{S}^n} \Big(1+(\frac{n}{2}-\gamma) w+ (n-2\gamma)^2 f\Big)\Big(P_\gamma(1)+(\frac{n}{2}-\gamma) P_\gamma(w)+ (n-2\gamma)^2 P_\gamma(f)\Big)\,\ud \mu_{g_0} \\&
=\dashint_{\mathbb{S}^n} \Big(Y(n,\gamma)+(n-2\gamma) Y(n,\gamma) w +2(n-2\gamma)^2 Y(n,\gamma) f +(\frac{n}{2}-\gamma)^2wP_\gamma w)\,\ud \mu_{g_0} \\&\quad +O((n-2\gamma)^3\Big),
\end{align*}
where we have used the self-adjointness  of $P_\gamma$ and $P_\gamma(1)=Y(n,\gamma)$. We also see that
\begin{align*}
&Y(n, \gamma)\dashint_{\mathbb{S}^n}|v|^{2}\,\ud \mu_{g_0} \\& = Y(n, \gamma)\dashint_{\mathbb{S}^n}\Big(1+(n-2\gamma)w+2(n-2\gamma)^2 f +(\frac{n}{2}-\gamma)^2 w^2 +O((n-2\gamma)^3)\Big)\,\ud \mu_{g_0} .
\end{align*}
It follows that
\begin{align*}
&\dashint_{\mathbb{S}^n}vP_{\gamma}(v)\,\ud \mu_{g_0} -  Y(n, \gamma)\dashint_{\mathbb{S}^n}|v|^{2}\,\ud \mu_{g_0}  \\&
=(\frac{n}{2}-\gamma)^2 \dashint_{\mathbb{S}^n} wP_\gamma w  \,\ud \mu_{g_0}  +O((n-2\gamma)^3).
\end{align*}
Let $w=\sum_{k=0}^\infty Y^{(k)}$, where $Y^{(k)}$ are spherical harmonics of degree $k$. Hence,
\begin{align*}
\dashint_{\mathbb{S}^n} wP_\gamma w  \,\ud \mu_{g_0}
&=\sum_{k=0}^\infty  \frac{\Gamma(k+\frac{n}{2}+\gamma)}{\Gamma(k+\frac{n}{2}-\gamma)} \dashint_{\mathbb{S}^n}|Y^{(k)}|^2  \,\ud \mu_{g_0}
\\&
\to \sum_{k=1}^\infty  \frac{\Gamma(k+n)}{\Gamma(k)}\dashint_{\mathbb{S}^n}|Y^{(k)}|^2  \,\ud \mu_{g_0}  =\dashint_{\mathbb{S}^n} wP_{n/2}w \,\ud \mu_{g_0}
\end{align*}
as $\gamma \to \frac{n}{2}$, where we have used \eqref{eigenvalue} in the first identity, the definition of $P_{n/2}$ in the second one and have used the smoothness of $w$ to ensure the convergence.  Therefore, \eqref{eq:right} follows.

Proposition \ref{prop:main} is proved.  \end{proof}

\section{The CR spheres setting}

Following T. Branson, L. Fontana and C. Morpurgo \cite{BFM},  we let  $\mathcal{H}_{j,k}$ be the space of harmonic polynomials of bidegree $(j,k)$ on CR sphere $\mathbb{S}^{2n+1}$, $j,k=0,1,\dots$; such spaces make up for the standard decomposition of $L^2$ into $U(n+1)$-invariant and irreducible subspaces, where $n\ge 1$. For $0<d<Q:=2n+2$, let $\mathcal{A}_{d}$ be the intertwining operator of order $d$ on CR sphere $\mathbb{S}^{2n+1}$,  characterized  by
\be
\mathcal{A}_d Y^{(j,k)}= \lda_j(d)\lda_k (d) Y^{(j,k)}, \quad \lda_j(d)=\frac{\Gamma(j+\frac{Q+d}{4})}{\Gamma(j+\frac{Q-d}{4})}
\ee
for every $Y^{(j,k)}\in \mathcal{H}_{j,k} $.  When $d=2$, it gives  the CR invariant sub-Laplacian, see D. Jerison and J.M. Lee \cite{JL}. 

One can define the operator $\mathcal{A}_{Q}  :=\lim_{d\to Q}\mathcal{A}_{d}$. The kernel of this operator is
the space of CR-pluriharmonic functions on   $\mathbb{S}^{2n+1}$ given by 
\[
\mathcal{P}:= \bigoplus_{j>0}(\mathcal{H}_{j,0} \bigoplus \mathcal{H}_{0,j}) \bigoplus \mathcal{H}_{0,0}.
 \]
It was discussed in \cite{BFM} that $\mathcal{A}_Q$ is not a suitable operator which could be used to conclude a conformally invariant Moser-Trudinger-Onofri inequality. In  \cite{BFM}, the authors defined the operator $\mathcal{A}_{Q}'$ acting on the CR-pluriharmonic functions  with
\be \label{eq:A'}
\mathcal{A}_{Q}' F= \prod_{\ell=0}^n(\frac{2}{n}\mathcal{L}+ \ell) F=\lim_{d\to Q}\frac{1}{\lda_0(d)} \mathcal{A}_d F, \quad \forall~ F\in C^\infty(\mathbb{S}^{2n+1}) \cap \mathcal{P},
\ee
where $\mathcal{L}=\mathcal{A}_2-\frac{n^2}{4}$ is the sub-Laplacian operator satisfying
\[
\mathcal{L} Y^{(j,k)}=(jk+\frac{n}{2}j +\frac{n}{2}k) Y^{(j,k)} \quad \mbox{for all }j,k\ge 0.
\]   The limit in the second equality of \eqref{eq:A'} is uniform, see Proposition 1.2 of \cite{BFM}. The sharp Moser-Trudinger-Onofri inequality on CR $\mathbb{S}^{2n+1}$ proved by \cite{BFM} asserts that
\be\label{eq:2-1}
\frac{n!}{Q}\ln \dashint_{\mathbb{S}^{2n+1}} e^{QF} \le \dashint_{\mathbb{S}^{2n+1}} F \mathcal{A}'_Q F + n! \dashint_{\mathbb{S}^{2n+1}} F \quad \mbox{for}~ F\in C^\infty(\mathbb{S}^{2n+1}) \cap \mathcal{P}.
\ee
(It is called Beckner-Onofri inequality in \cite{BFM}.) By duality,  the sharp Hardy-Littlewood-Sobolev inequality on CR $\mathbb{S}^{2n+1}$ due to R. Frank and E. Lieb \cite{FL} yields that
\be\label{eq:2-2}
\lda_0(d)^2  \left(\dashint_{\mathbb{S}^{2n+1} }|v|^{\frac{2Q}{Q-d}}\right)^{\frac{Q-d}{Q}}\leq
\dashint_{\mathbb{S}^{2n+1}}v\mathcal{A}_{d}(v)\quad \mbox{for }v\in C^\infty(\mathbb{S}^{2n+1}).
\ee

\begin{prop}\label{prop:main'} For any $F\in C^{\infty}(\mathbb{S}^{2n+1})\cap \mathcal{P}$, let $v= e^{\frac{Q-d}{2} F}$. Denote
\[
LHS_d:=  \frac{4}{(Q-d)^2} \lda_0(d)  \left[ \left( \dashint_{\mathbb{S}^{2n+1}}|v|^{\frac{2Q}{Q-d}}\right)^{\frac{Q-d}{Q}}- \dashint_{\mathbb{S}^{2n+1}}|v|^{2} \right]
\]
and
\[
RHS_d :=  \frac{4}{(Q-d)^2} \lda_0(d)^{-1} \left[ \dashint_{\mathbb{S}^{2n+1}}v\mathcal{A}_{d}(v)-  \lda_0(d)^2 \dashint_{\mathbb{S}^{2n+1}}|v|^{2}  \right].
\]
Then
\be  \label{eq:left-1}
\lim_{d\to Q} LHS_d= \frac{n!}{Q}  \ln \dashint_{\mathbb{S}^{2n+1}} e^{Q(F-\bar F)}
\ee
and
\be \label{eq:right-1}
\lim_{d\to Q} RHS_d= \dashint_{\mathbb{S}^{2n+1}} F \mathcal{A}'_{Q} F ,
\ee
where $\bar F$ is the average of $F$ over $\mathbb{S}^{2n+1}$.

\end{prop}

\begin{proof} Note that \begin{align*}
 &   \left(\dashint_{\mathbb{S}^{2n+1}}|v|^{\frac{2Q}{Q-d}}\right)^{\frac{Q-d}{Q}}- \dashint_{\mathbb{S}^{2n+1}}|v|^{2} \\
  &= \left(\dashint_{\mathbb{S}^{2n+1}}e^{QF}\right)^{\frac{Q-d}{Q}}-1- \dashint_{\mathbb{S}^{2n+1}}(e^{(Q-d)F}-1). \\&
 \end{align*}
 Then by L'H\^opital's rule
 \begin{align*}
 \lim_{d\to n/2} LHS_d &=\Gamma(n+1)  \left(\frac{1}{Q} \ln  \dashint_{\mathbb{S}^{2n+1}}e^{QF} - \dashint_{\mathbb{S}^{2n+1}} F \right) \\&
 =\frac{n!}{Q} \ln \dashint_{\mathbb{S}^{2n+1}} e^{Q(F-\bar F)}.
 \end{align*}
 Therefore, \eqref{eq:left-1} is proved.

To prove \eqref{eq:right-1}, using the Taylor expansion of the exponential function, we write
\[
v=e^{\frac{Q-d}{2} F}= 1+\frac12(Q-d) F+ (Q-d)^2 f,
\] where $f=\frac{1}{8} F^2 \int_0^1(1-s)e^{\frac{Q-d}{2} F s}\,\ud s\in C^\infty(\mathbb{S}^{2n+1})$ is uniformly bounded in $C^{4n}$ norm  as $d \to Q$.
Then we see that
\begin{align*}
&\dashint_{\mathbb{S}^{2n+1}}v\mathcal{A}_d(v) \\&=\dashint_{\mathbb{S}^{2n+1}} (1+\frac12(Q-d) F+ (Q-d)^2 f)(\mathcal{A}_d(1)+\frac12(Q-d) \mathcal{A}_d(F)+ (Q-d)^2 \mathcal{A}_d(f))\\&
=\dashint_{\mathbb{S}^{2n+1}}\Big (\lda_0(d)^2+(Q-d) \lda_0(d)^2 F +2(Q-d)^2 \lda_0(d)^2 f \\&\quad  +\frac14(Q-d)^2F\mathcal{A}_d F + (Q-d)^3 f \mathcal{A}_d F\Big) +O((Q-d)^4),
\end{align*}
where we have used the self-adjointness  of $\mathcal{A}_d$ and $\mathcal{A}_d(1)=\lda_0(d)^2$. We also see that
\begin{align*}
&\lda_0(d)^2\dashint_{\mathbb{S}^{2n+1}}|v|^{2}\\& = \lda_0(d)^2\dashint_{\mathbb{S}^{2n+1}}\Big(1+(Q-d)F+2(Q-d)^2 f +\frac12(Q-d)^2 F^2 +O((Q-d)^3)\Big).
\end{align*}
It follows that
\begin{align*}
&\dashint_{\mathbb{S}^{2n+1}}v\mathcal{A}_d(v)-  \lda_0(d)^2\dashint_{\mathbb{S}^{2n+1}}|v|^{2} \\&
=\frac14(Q-d)^2  \dashint_{\mathbb{S}^{2n+1}} (F\mathcal{A}_d F+ 4(Q-d)f \mathcal{A}_d F)  +O((Q-d)^4).
\end{align*}
By the second equality of \eqref{eq:A'}, \eqref{eq:right-1} follows immediately. 

Therefore, Proposition \ref{prop:main'} is proved.

\end{proof}

Similarly, we immediately  obtain

\begin{thm} \label{thm'} We can derive the sharp Moser-Trudinger-Onofri inequality \eqref{eq:2-1} from the sharp Sobolev inequality  \eqref{eq:2-2} by sending $d\to Q$.

\end{thm}

\bigskip

\noindent School of Mathematical sciences, Beijing Normal University\\
Beijing 100875, China\\[1mm]
Email: \textsf{jx@bnu.edu.cn}

\end{document}